\documentclass[12pt]{amsart}

\usepackage{amsmath,amssymb}    % need for subequations
\usepackage{graphicx}   % need for figures
\usepackage{verbatim}   % useful for program listings
\usepackage{color}      % use if color is used in text
\usepackage{subfigure}  % use for side-by-side figures
\usepackage{hyperref}   % use for hypertext links, including those to external documents and URLs

\usepackage{latexsym}
\usepackage{epsfig}
\usepackage{amsfonts}
\usepackage{enumerate}
\usepackage{times}

\newtheorem{prop}{Proposition}[section]
\newtheorem{lem}[prop]{Lemma}

\newtheorem{cond}[subsection]{Condition}

\theoremstyle{definition}

\theoremstyle{remark}

\theoremstyle{remark}

\newtheorem{remark}[subsection]{Remark}

\numberwithin{equation}{section}

\renewcommand{\P}{\mathbb{P}}
\renewcommand{\O}{\mathcal{O}}

\renewcommand{\i}{\hat{i}}
\renewcommand{\j}{\hat{j}}
\renewcommand{\k}{\hat{k}}
\newcommand{\pr}{\partial}

\newcommand{\inn}{\{i_1,\dots,i_n\}}

\newcommand{\vin}{v_{\{i_1,\dots,i_n\}}}

\newcommand{\nin}{n_{\{i_1,\dots,i_n\}}}

\newcommand{\ra}{\right\rangle}
\newcommand{\la}{\left\langle}

\newcommand{\Mbar}{\overline{\M}}

\newcommand{\M}{\mathcal{M}}
\newcommand{\C}{\mathcal{C}}

\newcommand{\D}{\mathcal{D}}

\newcommand{\sL}{\mathcal{L}}

\newcommand{\bc}{\mathbf{c}}

\def\<{\left\langle}
\def\>{\right\rangle}
\begin{document}

\title{On computations of genus zero two-point descendant Gromov-Witten invariants}
\author{Amin Gholampour}
\address{Department of Mathematics\\ University of Maryland\\ 1301 Mathematics Building\\ College Park\\ MD 20742\\ USA}
\email{amingh@math.umd.edu}
\author{Hsian-Hua Tseng}
\address{Department of Mathematics\\ Ohio State University\\ 100 Math Tower, 231 West 18th Ave. \\ Columbus \\ OH 43210\\ USA}
\email{hhtseng@math.ohio-state.edu}

\date{\today}

\begin{abstract}
We present a method of computing genus zero two-point descendant Gromov-Witten invariants via one-point invariants. We apply our method to recover some of calculations of Zinger and Popa-Zinger, as well as to obtain new calculations of two-point descendant invariants.
\end{abstract}

\maketitle

\section{Introduction}

Let $X$ be a smooth proper Deligne-Mumford $\mathbb{C}$-stack with projective coarse moduli space.  Genus $0$ two-point descendant Gromov-Witten invariants of $X$ are invariants of the following kind:
\begin{equation}\label{2pt_inv}
\<a\psi^k, b\psi^l\>_{0, 2, \beta}^X:=\int_{[\Mbar_{0,2}(X, \beta)]^{vir}}ev_1^*(a)\psi_1^k ev_2^*(b)\psi_2^l,
\end{equation}
where $a, b\in H^*(IX)$, $k, l\in \mathbb{Z}_{\geq 0}$, and $ev_1, ev_2: \Mbar_{0,2}(X, \beta)\to IX$ are the evaluation maps. We refer to \cite{agv1} for the basics of the construction of Gromov-Witten invariants for Deligne-Mumford stacks. 

Recently exact computations of genus $0$ two-point descendant Gromov-Witten invariants received much attention becaue of mirror symmetry for genus $1$ and open Gromov-Witten invariants. In the case $X=\mathbb{P}^n$ a formula for the invariants (\ref{2pt_inv}) is proved in \cite{z}. Formulas for variants of (\ref{2pt_inv}) involving twists by Euler class and direct sums of line bundles, in the sense of \cite{cg}, are also proven in \cite{z} and \cite{pz} in the toric setting. More recently a formula for the invariants (\ref{2pt_inv}) for compact symplectic toric manifolds is proven in \cite{p}. The proofs in \cite{z, pz, p} follow a strategy that is similar to the one used by Givental in his computation of genus $0$ one-point descendant invariants \cite{g1, g2}. More precisely, a generating function of invariants (\ref{2pt_inv}) is proven by virtual localization to satisfy certain recursion relations and certain regularity conditions. The localization computations needed in \cite{z, pz, p} are somewhat involved.

The purpose of this paper is to discuss a simpler method for explicitly computing (\ref{2pt_inv}). This method is based on a known fact in topological field theory, which relates two-point descendant invariants (\ref{2pt_inv}) to one-point descendant invariants, see (\ref{2pt_to_1pt}). We explain this method in detail in Section \ref{the_method}. In Section \ref{sec:app} we apply this method to compute two-point descendant invariants for several classes of examples.

\subsection*{Convention}
We work over the field of complex numbers. Cohomology groups are taken with rational coefficients. In this paper we only consider cohomology in even degrees.

%\section*{Acknowledgment} 

\section{Method of computation}\label{the_method}
In this Section we present our method for computing two-point descendant invariants (\ref{2pt_inv}). We work in the more general context of {\em twisted orbifold Gromov-Witten theory}, as constructed in \cite{t}.  We briefly recall this theory, following \cite{t} (but using somewhat different notations).
\subsection{Set-up}\label{twisted_theory_setup} 
 Let $X$ be a smooth proper Deligne-Mumford $\mathbb{C}$-stack with projective coarse moduli space. Let $V\to X$ be a complex vector bundle, and $\bc(-)$ a multiplicative invertible characteristic class of vector bundles. Given two integers $g,n\geq 0$ and $\beta\in H_2(X,\mathbb{Z})$, let $\Mbar_{g,n}(X, \beta)$ be the moduli stack of $n$-pointed genus $g$ degree $\beta$ orbifold stable maps to $X$. For each $i=1,...,n$ there is an evaluation map $ev_i: \Mbar_{g,n}(X, \beta)\to IX$, taking values in the inertia stack $IX$ of $X$. Let $\pi: \C\to \Mbar_{g,n}(X,\beta)$ be the universal curve and $f: \C\to X$ the universal orbifold stable map. A key ingredient in the construction of the twisted theory is the following element in the $K$-theory:
\begin{equation}
V_{g,n,\beta}:= R\pi_*f^*V\in K^0(\Mbar_{g,n}(X,\beta)).
\end{equation} 
The $(\bc, V)$-twisted orbifold Gromov-Witten invariants of $X$ are defined by 
\begin{equation}\label{twisted_inv_defn}
\<a_1\psi^{k_1},..., a_n\psi^{k_n}\>_{g,n,\beta}^{X,(\bc, V)}:=\int_{[\Mbar_{g,n}(X, \beta)]^{vir}}\bc(V_{g,n,\beta})\prod_{i=1}^n ev_i^*(a_i)\psi_i^{k_i}.
\end{equation}
Here $k_1,...,k_n\geq 0$ are integers, $a_1,...,a_n\in H^*(IX)$, $$[\Mbar_{g,n}(X,\beta)]^{vir}\in H_*(\Mbar_{g,n}(X,\beta),\mathbb{Q})$$ is the virtual fundamental class, and $\psi_i\in H^2(\Mbar_{g,n}(X, \beta),\mathbb{Q})$ are the descendant classes.
 
\subsection{Reduction to one-point descendant}

Let $\tau\in H^*(IX)$. Consider the linear map $$R(\tau; z_1,z_2): H^*(IX)\to H^*(IX)[[z_1^{-1}, z_2^{-1}]]$$  defined by requiring that for $a, b\in H^*(IX)$ we have 
\begin{equation}
(a, R(\tau; z_1,z_2)(b))_{(\bc, V)}:=(a,b)_{(\bc, V)}+\sum_{\beta}\sum_{n}\frac{Q^\beta}{n!}\<\frac{a}{z_1-\psi}, \tau,...,\tau, \frac{b}{z_2-\psi}\>^{X, (\bc, V)}_{0,n+2, \beta},
\end{equation}
where $(-,-)_{(\bc, V)}$ is the $(\bc, V)$-orbifold Poincar\'e pairing of $X$, as defined in \cite[Section 3.2]{t}, and $Q^\beta$ is an element in the Novikov ring. We can consider $R(\tau; z_1,z_2)$ as a generating function of genus $0$ two-point twisted descendant invariants. 

Consider the linear map $$S(\tau; z): H^*(IX)\to H^*(IX)[[z^{-1}]]$$ defined by requiring that for $a, b\in H^*(IX)$ we have 
\begin{equation}
(a, S(\tau; z)(b))_{(\bc, V)}:= (a,b)_{(\bc,V)}+ \sum_{\beta}\sum_{n}\frac{Q^\beta}{n!}\<a, \tau,...,\tau, \frac{b}{z-\psi}\>^{X,(\bc,V)}_{0,n+2, \beta}.
\end{equation}
Likewise we can consider $S(\tau; z)$ as a generating function of genus $0$ one-point twisted descendant invariants.

\begin{prop}
\begin{equation}\label{2pt_to_1pt}
R(\tau; z_1,z_2)=\frac{1}{z_1+z_2}\left(S(\tau;z_1)^*S(\tau;z_2)-Id\right).
\end{equation}
Here the superscript $*$ indicates the adjoint with respect to the pairing $(-,-)_{(\bc, V)}$. 
\end{prop}

This Proposition gives a relationship between the linear maps $R(\tau;z_1,z_2)$ and $S(\tau;z)$. This Proposition is not new. For the sake of completeness we present a proof of this Proposition below.

\begin{proof}
The proof of this Proposition is a straightforward application of the argument that proves WDVV equations.

By string equation, we have 
\begin{equation}\label{2to1_pf_1}
\begin{split}
&\<\frac{a}{z_1-\psi}, \tau,...,\tau, 1, \frac{b}{z_2-\psi}\>_{0,n+3,\beta}^{X,(\bc,V)}\\
=&\frac{1}{z_1}\<\frac{a}{z_1-\psi}, \tau,...,\tau, \frac{b}{z_2-\psi}\>_{0,n+2,\beta}^{X,(\bc,V)}+\<\frac{a}{z_1-\psi}, \tau,...,\tau, \frac{b}{z_2-\psi}\>_{0,n+2,\beta}^{X,(\bc,V)}\frac{1}{z_2}\\
=&\left(\frac{1}{z_1}+\frac{1}{z_2}\right)\<\frac{a}{z_1-\psi}, \tau,...,\tau, \frac{b}{z_2-\psi}\>_{0,n+2,\beta}^{X,(\bc,V)}.
\end{split}
\end{equation}
In the exceptional case $(n, \beta)=(0,0)$ we have 
\begin{equation}\label{2to1_pf_1.5}
\<\frac{a}{z_1-\psi},1,\frac{b}{z_2-\psi}\>_{0,3,0}^{X, (\bc,V)}=\frac{1}{z_1}\frac{1}{z_2}(a,b)_{(\bc, V)}.
\end{equation}
Let $\{\phi_\alpha\}\subset H^*(IX)$ be an additive basis, and $\{\phi^\alpha\}\subset H^*(IX)$ be the dual basis with respect to the pairing $(-,-)_{(\bc, V)}$. The rational equivalence of boundary divisors in $\Mbar_{0,4}$ used in the proof of the WDVV equation gives the following 
\begin{equation}\label{2to1_pf_2}
\begin{split}
&\sum_{n_1+n_2=n, \beta_1+\beta_2=\beta} \sum_{\alpha}\<\frac{a}{z_1-\psi}, \tau,...,\tau, 1, \phi_\alpha\>_{0,n_1+3, \beta_1}^{X,(\bc,V)}\<\phi^\alpha, 1,\tau,...,\tau, \frac{b}{z_2-\psi}\>_{0, n_2+3, \beta_2}^{X,(\bc,V)}\\
=& \sum_{n_1+n_2=n, \beta_1+\beta_2=\beta} \sum_{\alpha}\<\frac{a}{z_1-\psi}, \tau,...,\tau, \frac{b}{z_2-\psi}, \phi_\alpha\>_{0,n_1+3, \beta_1}^{X,(\bc,V)}\<\phi^\alpha, 1,\tau,...,\tau,1\>_{0, n_2+3, \beta_2}^{X,(\bc,V)}\\
=&\<\frac{a}{z_1-\psi}, \tau,...,\tau, \frac{b}{z_2-\psi}, \sum_\alpha (\phi^\alpha, 1)_{(\bc, V)}\phi_\alpha\>_{0,n+3, \beta}^{X,(\bc,V)} \quad\quad (\text{by string equation})\\
=& \<\frac{a}{z_1-\psi}, \tau,...,\tau, \frac{b}{z_2-\psi}, 1\>_{0,n+3, \beta}^{X,(\bc,V)}.
\end{split}
\end{equation}
Again by string equation, we have 
\begin{equation}\label{2to1_pf_3}
\begin{split}
&\<\frac{a}{z_1-\psi}, \tau,...,\tau, 1, \phi_\alpha\>_{0,n_1+3, \beta_1}^{X,(\bc,V)}= \frac{1}{z_1}\<\frac{a}{z_1-\psi}, \tau,...,\tau, \phi_\alpha\>_{0,n_1+2, \beta_1}^{X,(\bc,V)},\\
&\<\phi^\alpha, 1,\tau,...,\tau, \frac{b}{z_2-\psi}\>_{0, n_2+3, \beta_2}^{X,(\bc,V)}=\frac{1}{z_2} \<\phi^\alpha, \tau,...,\tau, \frac{b}{z_2-\psi}\>_{0, n_2+2, \beta_2}^{X,(\bc,V)}, 
\end{split}
\end{equation}
with the exception that
\begin{equation}\label{2to1_pf_3.5}
\<\frac{a}{z_1-\psi}, 1, \phi_\alpha\>_{0,3, 0}^{X,(\bc,V)}=\frac{1}{z_1}(a,\phi_\alpha)_{(\bc, V)}, \quad \<\phi^\alpha, 1, \frac{b}{z_2-\psi}\>_{0, 3, 0}^{X,(\bc,V)}=\frac{1}{z_2}(\phi^\alpha,b)_{(\bc, V)}.
\end{equation}
Combining (\ref{2to1_pf_1})--(\ref{2to1_pf_3.5}) and summing over all possible values of $n$ and $\beta$, we get 
\begin{equation}
\left(\frac{1}{z_1}+\frac{1}{z_2}\right)R(\tau;z_1,z_2)+\frac{1}{z_1z_2}Id=\frac{1}{z_1}\frac{1}{z_2}S(\tau;z_1)^*S(\tau;z_2),
\end{equation}
which is (\ref{2pt_to_1pt}).
\end{proof}

\begin{remark}
\hfill
\begin{enumerate}
\item
To avoid notational complications, (\ref{2pt_to_1pt}) is not stated for equivariant Gromov-Witten invariants. However it is clear from the proof that (\ref{2pt_to_1pt}) is also valid in equivariant Gromov-Witten theory. 

\item
It is easy to see that when $X$ is a compact symplectic toric manifold, the equivariant version of (\ref{2pt_to_1pt}) recovers \cite[Theorem 4.5]{p}.

\item
A formula for genus $0$ {\em multi-point} Gromov-Witten invariants of $\mathbb{P}^n$ is proved in \cite{z2}. It is clear that the method used to prove (\ref{2pt_to_1pt}) can be applied recursively to prove formulas for genus $0$ multi-point Gromov-Witten invariants of more general target space $X$. We do not pursue this here.

\end{enumerate}

\end{remark}

Equation (\ref{2pt_to_1pt}) expresses $R(\tau;z_1,z_2)$ in terms of $S(\tau;z)$. This reduces the computation of $R(\tau;z_1,z_2)$ to the computation of $S(\tau;z)$. 

\subsection{One-point invariants}
As mentioned above, $S(\tau;z)$ can be considered as a generating function of one-point twisted descendant invariants. When considering one-point twisted descendant invariants, the so-called twisted 
$J$-function $J_{X,(\bc, V)}(\tau;z)$ plays an important role: 
\begin{equation}
J_{X, (\bc, V)}(\tau;z):=z+\tau+ \sum_{\beta}\sum_{n}\frac{Q^\beta}{n!}\< 1,\tau,...,\tau, \frac{\phi_\alpha}{z-\psi}\>^{X,(\bc,V)}_{0,n+2, \beta}\phi^{\alpha}.
\end{equation}
It is easy to see that $J_{X,(\bc, V)}(\tau;z)=S(\tau;z)^*(1)$. In other words, the $J$-function is the ``first column'' of $S(\tau;z)^*$. 

In various cases of $X, (\bc, V)$ the {\em small} $J$-function $$J_{X,(\bc, V)}(\tau;z)|_{\tau\in H^0(X)\oplus H^2(X)}$$ is explicitly known. For example, when $X=\mathbb{P}^{n}$ and $(\bc, V)$ is vacuous, the small $J$-function is given by 
\begin{equation}\label{J_Pn}
J_{\mathbb{P}^{n}}(\tau=t_01+tP;z)=ze^{(t_01+Pt)/z}\sum_{d\geq 0}\frac{Q^de^{dt}}{\prod_{k=1}^d(P+kz)^{n+1}},
\end{equation}
where $1\in H^0(\mathbb{P}^{n})$ and $P\in H^2(\mathbb{P}^{n})$ is the hyperplane class. It is evident that $$J_{\mathbb{P}^{n}}(\tau=t_01+tP;z)$$ is not the whole $S(\tau;z)$. But in this case one can check that 
\begin{equation}\label{derivative}
(z\partial/\partial t)^j J_{\mathbb{P}^{n}}(t_01+tP;z)=z\nabla_{P^j}J_{\mathbb{P}^{n}}(\tau;z)|_{\tau=t_01+tP}.
\end{equation}
One can also check that the derivative of the full $J$-function $J_{\mathbb{P}^{n}}$ along the direction of $P^j\in H^{2j}(\mathbb{P}^{n})$,  $\nabla_{P^j}J_{\mathbb{P}^{n}}(\tau;z)$, gives other columns of $S(\tau;z)^*$. Thus by (\ref{J_Pn}) and (\ref{derivative}) we can explicitly compute $S(\tau;z)^*$ for $\tau=t_01+tP$. 

The example above suggests the following way to compute $R(\tau;z_1,z_2)|_{\tau\in H^0(X)\oplus H^2(X)}$.  Suppose that the small $J$-function $J_{X,(\bc, V)}(\tau;z)|_{\tau\in H^0(X)\oplus H^2(X)}$ is explicitly known, and suppose that we can obtain all columns of $S(\tau;z)^*|_{\tau\in H^0(X)\oplus H^2(X)}$ by successive differentiations along $H^2(X)$, then we can obtain an explicit formula for $S(\tau;z)^*|_{\tau\in H^0(X)\oplus H^2(X)}$. Using (\ref{2pt_to_1pt}) we get an explicit formula for $R(\tau;z_1,z_2)|_{\tau\in H^0(X)\oplus H^2(X)}$. Finally, the desired two-point twisted descendant invariants are extracted from $R(\tau;z_1,z_2)|_{\tau\in H^0(X)\oplus H^2(X)}$, after applying string and divisor equations. In the next subsection we set up this computation scheme in more details.

\subsection{Computation scheme}
Let $X, (\bc, V)$ be as in Section \ref{twisted_theory_setup}. Suppose we can find the elements 
$$\{v_i\}_{i=1,\dots,N}$$ in of $H^{*}(IX)$ with the following properties:
\begin{enumerate}
\item There exists a permutation $\hat{1},\dots, \hat{N}$ of $1,\dots,N$ such that the paring satisfies 
$$(v_{\i},v_j)_{(\bc, V)}=m_{i}\delta_{\i j}$$ for a nonzero $m_i$.
\item The restriction of  $\nabla_{v_j}J_{X, (\bc, V)}(\tau;z)$ to $H^2(X)$ is known for each $v_j$. 
\end{enumerate}
Let $S_{ij}(\tau)$ be the component of $\nabla_{v_j}J_{X, (\bc, V)}(\tau;z)|_{\tau\in H^2(X)}$ along $v_i$.

We know that $\nabla_{v_j}J_{X, (\bc, V)}(\tau;z)$ is the $j$-th column of $S(\tau;z)^*$. So for $t\in H^2(X)$, the  $(i,j)$-entry of $S(t;z)^*$ is 
$$(\frac{v_{\i}}{m_{\i}},S(t;z)^*(v_j))_{(\bc,V)}=S_{ij}(t).$$ 
Hence the $(i,j)$-entry of $S(t;z)$ is 
$$(\frac{v_{\i}}{m_{\i}},S(t;z)(v_j))_{(\bc,V)}=\frac{1}{m_{\i}}(v_j,S(t;z)^*(v_{\i}))_{(\bc,V)}=\frac{m_j}{m_i}S_{\j\i}(t).$$ 
From this we get the $(i,j)$-entry of the $R(t;z_1,z_2)$ to be
\begin{equation} \label{equ:kj-entry}
\frac{1}{z_1^{-1}+z_2^{-1}}\left(\sum_k  \frac{m_k}{m_j}S_{ik}S_{\j\k}-\delta^i_j\right).
\end{equation}
After setting $t=0$ in (\ref{equ:kj-entry}), the coefficient of $Q^\beta$ gives the desired two-point $(\bc, V)$-twisted Gromov-Witten invariant
$$\left\langle \frac{v_{\i}}{m_i(z_1-\psi)},\frac{v_k}{z_2-\psi} \right\rangle_{0,2,\beta}^{X,(\bc, V)}.$$

In the next section we investigate some cases that one can find a cohomology basis with the properties above by using mirror theorems. As in the example of $\P^{n-1}$ above, in all of our applications we can find a collection $\{D_1, D_2,...\}$ of first-order linear differential operators with differentiations only along directions in $H^2(X)$ such that for each $1 \le j \le N$ there exists $i_1,\dots, i_n$ such that 
$$\textbf{($\star$)} \quad\quad z\nabla_{v_j}J_{X, (\bc, V)}(\tau;z)|_{\tau\in H^2(X)}=zD_{i_1}\circ  \cdots \circ zD_{i_n} I_{X, (\bc, V)}$$
after the change of variables in the mirror theorem. Here $I_{X, (\bc, V)}$ denotes the Givental $I$-function. Since the $J$-function takes the form $J_{X, (\bc, V)}(\tau;z)=z+\tau+O(z^{-1})$, in order for \textbf{($\star$)} to be true we need to verify the following condition in all our applications.
%The following assumption is needed in order to implement the computation scheme.
\begin{cond} \label{condition}
For each $j$ and $i_1,\dots i_n$ as above the only positive powers of $z$ appearing in $zD_{i_1}\circ \cdots \circ zD_{i_{n-1}}I_{X, (\bc, V)}$ is $Az$ for some $A\in H^*(IX)$.
\end{cond}

\section{Some applications}\label{sec:app}
In this Section we implement the aforementioned computation scheme for weighted projective spaces and some toric manifolds.

\subsection{Weighted projective space $X=\mathbb{P}(w_0,w_1,...,w_n)$} \label{sec:weighted projective spaces}
In this section we mostly follow the notation in \cite{cclt}. Let $P\in H^2(X)$ be the hyperplane class, and let $N = w_0 +\cdots+w_n$. Denote by $\la a\ra$ the fractional part of the rational number $a$. The small $J$-function of the weighted projective spaces was computed in \cite[Theorem 1.7]{cclt}:

\begin{equation*} J_X(t;z)=z e^{Pt/z}\sum_{d \ge 0 ; \; \langle d \rangle \in F}\frac{e^{dt}Q^{d}}{\prod_{i=0}^n\prod_{0<b\le dw_i;  \; \langle b \rangle=\langle dw_i \rangle} (w_iP+bz)}1_{ \langle d \rangle} , 
\end{equation*}
where 
$$F =\{k/w_i :0\le k<w_i \;\;\text{and}\;\; 0\le i \le n\},$$ and $c_1,...,c_N$ are defined to be the sequence obtained by arranging the terms$$\frac{0}{w_0}, \frac{1}{w_0},\dots,\frac{w_0-1}{w_0}, \frac{0}{w_1}, \frac{1}{w_1},\dots,\frac{w_1-1}{w_1},\dots,\frac{0}{w_n}, \frac{1}{w_n},\dots,\frac{w_n-1}{w_n}$$ in increasing order. The connected components of $IX$ are indexed by the elements of $F$. For any $f\in F$ let $1_f$ be the fundamental class of the corresponding component of $IX$. By \cite[Lemma 5.1]{cclt} there exists a basis $\mathcal{B}=\{v_1,\dots, v_N\}$ for $H^*(IX)$ given by $v_j=\sigma_jP^{r_j}1_{c_j}$ where $$\sigma_j=\frac{\prod_{m: c_m<c_j}(c_j-c_m)}{\prod_{i=0}^n\prod_{\tiny \begin{array}{r}b: \langle b \rangle=\langle c_jw_i \rangle\\ 0<b\le c_jw_i\end{array}} b},$$ and $r_j=\#\{i| i<j ,c_i=c_j\}$. Define $$d_j=\#\{i| c_i=c_j\},$$  $$m_j=\prod_{\{i|c_j w_i \in \mathbb{Z}\}}w_i,$$  then the dual basis of $\mathcal{B}$ is given by $\{v^1,\dots, v^N\}$ where $$v^j=\frac{m_j}{\sigma_j}P^{d_j-r_j}1_{\langle 1-c_j \rangle}=\frac{m_j}{\sigma_j\sigma_{\hat{j}}}v_{\hat{j}}.$$ Note that  we define $\hat{j}$ by the second equality above.

We know that $\nabla_{v_j}J_X(\tau;z)$ is the $j$th column of $S(\tau;z)^*$, and by \cite[Lemma 5.1]{cclt} there exist explicitly given linear differential operators $D_1,\dots, D_N$ such that $$\nabla_{v_j}J(\tau;z)|_{\tau=tP}=z^{-1}D_jJ(t;z).$$ So if we denote by $J_X(\tau;z)$ the component of the $J$-function along $v_k$, the  $(k,j)$-entry of $S(tP;z)^*$ is $$\langle v^k,S(tP;z)^*(v_j)\rangle=z^{-1}D_jJ_k(t;z).$$ So the $(k,j)$-entry of $S(tP;z)$ is \begin{align*}\langle v^k,S(tP;z)(v_j)\rangle=\langle v_j,S(tP;z)^*(v^k)\rangle&=\frac{m_k\sigma_j \sigma_{\hat{j}}}{m_{\hat{j}}\sigma_k\sigma_{\hat{k}}}\langle v^{\hat{j}},S(tP;z)^*(v_{\hat{k}})\rangle\\&=\frac{m_k\sigma_j \sigma_{\hat{j}}}{m_{\hat{j}}\sigma_k\sigma_{\hat{k}}}z^{-1}D_{\hat{k}}J_{\hat{j}}(t).\end{align*} From this the $(k,j)$-entry of the $R(t;z_1,z_2)$ is 

\begin{equation} \label{equ:kj-entryI}
\frac{1}{z_1^{-1}+z_2^{-1}}\left(\sum_i \frac{m_i\sigma_j \sigma_{\hat{j}}}{m_{\hat{j}}\sigma_i\sigma_{\hat{i}}}D_iJ_k(t;z_1)D_{\hat{i}}J_{\hat{j}}(t;z_2)-\delta^j_k\right).
\end{equation}
 After setting $t=0$ in (\ref{equ:kj-entryI}), the coefficient of $Q^d$ gives the desired 2-point descendant Gromov-Witten invariant $$\left\langle \frac{v^j}{z_1-\psi_1},\frac{v_k}{z_2-\psi_2} \right\rangle_{0,d}^X.$$

We know from \cite[Proof of Lemma 5.1]{cclt} that 
\begin{equation} \label{equ:DjJ}D_jJ(0;z)=z \sum_{d \ge 0 ; \; \langle d \rangle \in F}\frac{\prod_{m=1}^{j-1}(P+(d-c_m)z)}{\prod_{i=0}^n\prod_{0<b\le dw_i;  \; \langle b \rangle=\langle dw_i \rangle} (w_iP+bz)}1_{ \langle d \rangle} Q^{d-c_j}. 
\end{equation}
So in order to compute the right hand side of (\ref{equ:kj-entryI}), we need to read off the coefficients of $D_iJ$ and $D_{\hat{i}}J$ along the specific basis elements. For this we introduce the new variables $H_1, H_2$ and $x_1, x_2$ keeping tracks of powers of $P$ and the indices of $1_{\langle d \rangle}$ in $D_iJ$ and $D_{\hat{i}}J$, respectively. Now using (\ref{equ:kj-entryI}) and (\ref{equ:DjJ}), we can write:
For $d>0$ we get 
\begin{align*}
&\sum_{j,k=1}^N\langle \frac{v_j}{z_1-\psi_1},\frac{v_k}{z_2-\psi_2} \rangle_{0,d}\frac{m^2_j}{\sigma_j^2 \sigma^2_{\hat{j}}}H_1^{r_j}H_2^{r_k} x_1^{c_j} x_2^{c_k}=\frac{1}{z_1+z_2}\sum_{s=1}^N  \sum_{\tiny \begin{array}{c} d_1,d_2\ge 0 \\ \langle d_1 \rangle, \langle d_2 \rangle  \in F \\ d_1+d_2=d+c_s+c_{\hat{s}}\end{array} }
 \\
&\frac{m_s\prod_{m=1}^{s-1}(H_1+(d_1-c_m)z_1)\prod_{\hat{m}=1}^{\hat{s}-1}(H_2+(d_2-c_{\hat{m}})z_2)}{\sigma_s\sigma_{\hat{s}}\prod_{i=0}^n\prod_{\tiny \begin{array}{l}  0<b_1\le d_1w_s;  \; \langle b_1 \rangle=\langle d_1w_s \rangle \\ 0<b_2 \le d_2w_{\hat{s}};  \; \langle b_2 \rangle=\langle d_2w_{\hat{s}} \rangle \end{array}} (w_s H_1+b_1z_1)(w_{\hat{s}} H_2+b_2z_2)}x_1^{\langle d_1 \rangle }x_2^{\langle 1-\langle d_2 \rangle \rangle }.
\end{align*}
This specializes to \cite[Theorem 1]{z}.

\begin{remark}
Using the mirror theorems stated in \cite{cclt}, our method can be applied to compute the twisted 2-point Gromov-Witten invariants of a complete intersection inside a weighted projective space if it satisfies the Condition \ref{condition}. However, since the mirror theorem usually involves nontrivial change of variables the formulas we get are less explicit. 
\end{remark}

%%%%%%%%%%%%%%%%%%%%%%%%
\subsection{Toric manifolds}
In this Section we discuss applications of the aforementioned method to compute genus $0$ two-point descendant Gromov-Witten invariants of a smooth projective toric variety. 

Let $X$ be a smooth projective toric variety. The (small) $J$-function $J_X$ of $X$ is determined by the toric mirror theorem \cite{g1, g2, i}. How {\em explicitly} the $J$-function of $X$ is determined depends on $X$. If $X$ is Fano (i.e. $-K_X$ is ample), then $J_X$ is equal to the combinatorially defined {\em $I$-function} $I_X$. If $X$ is semi-Fano but not Fano (i.e. $-K_X$ is nef and big but not ample), then $J_X$ is equal to $I_X$ after a change of variables (the inverse mirror map) which is often given by power series with recursively determined coefficients. If $X$ is not semi-Fano, the situation is quite complicated. 

In Section \ref{subsec:toric_mir} we discuss how to use toric mirror theorems to compute the necessary generating function $S(\tau;z)$ for toric manifolds $X$. The outcome is not very explicitly, as it involves some recursively determined quantities. In Section \ref{subsec:fixed_pt} we discuss an approach to yield more explicit formulas in the toric Fano case.

It is worth mentioning that the discussions in this Section in principle works for toric Deligne-Mumford stacks as well, given the appropriate mirror theorem for them (see \cite{ccit}).

\subsubsection{Using mirror theorem in general}\label{subsec:toric_mir}
Let $X$ be a smooth projective toric manifold. According to \cite{g3}, the totality of genus $0$ Gromov-Witten invariants of $X$ can be encoded in a Lagrangian submanifold $\sL_X$ in a suitable symplectic vector space. Following \cite{g2} one can write down a cohomology-valued formal function $I_X(t;z)$ called the {\em I-function} of $X$. The toric mirror theorem in this generality (see \cite{i}) states that the family $$t\mapsto I_X(t;z), t\in H^2(X)$$ lies on $\sL_X$. By general properties of the Lagrangian submanifold $\sL_X$ (see \cite{g3}), this implies that $\sL_X$ (and consequently the genus $0$ Gromov-Witten theory of $X$) is determined by $I_X(t;z)$. On the other hand, the family $$\tau\mapsto J_X(\tau;z)$$ defined by the $J$-function also lies on $\sL_X$. 

Thus it is possible to determine $J_X$ from $I_X$. However in this generality the process of determining $J_X$ from $I_X$ involves {\em Birkhoff factorization}, as explained in \cite[Pages 29--30]{cg}. Moreover, for the computations of two-point Gromov-Witten invariants, we need to determine not only the $J$-function $J_X$, but also other columns of $S(\tau;z)^*$. To do this we need to use the fact that differentiation along any direction in the cohomology $H^*(X)$ can be expressed as a higher-order differential operator involving only differentiations along directions in $H^2(X)$ (this is true because $H^*(X)$ is multiplicatively generated by $H^2(X)$). To summarize, there exists differential operators $P_i, i=1, ..., \text{dim}\, H^*(X)$ involving only differentiations in $H^2(X)$ directions, satisfying the following property: Let $(P_i I_X(t;z))$ be the matrix whose columns are $P_iI_X$. Then there exists a matrix-valued formal series $B(\tau;z)$ in $z$ such that 
\begin{equation}\label{birkhoff_fac}
(P_iI_X(t;z))=S(\tau;z)^*B(\tau;z).
\end{equation}
We refer to \cite[Proposition 5.6]{llw} for more detailed discussions on this.

(\ref{birkhoff_fac}) together with (\ref{2pt_to_1pt}) allow us to express $R(\tau;z_1,z_2)$ as follows:
\begin{equation}\label{toric_2_to_1}
R(\tau;z_1, z_2)= \frac{1}{z_1+z_2}\left((P_iI_X(t;z_1))B(\tau;z_1)^{-1}(B(\tau;z_2)^*)^{-1}(P_iI_X(t;z_2))^*-Id\right).
\end{equation}

Unfortunately, the equation (\ref{toric_2_to_1}) above is not very explicit, because the differential operators, the Birkhoff factorizations, and the generalized mirror map $\tau=\tau(t)$ can only be determined recursively. It may be possible to produce recursive algorithms for computing two-point Gromov-Witten invariants using (\ref{toric_2_to_1}), but we do not pursue it here.

\subsubsection{Fixed point set method}\label{subsec:fixed_pt}
\newcommand{\oneton}{ \{1,\dots,N\}}

Let $X$ be an $n$-dimensional smooth toric variety whose toric fan is generated by the rays $r_1,\dots r_N$. In this section we mostly follow the notation in \cite{g2}. If $\{P_1,\dots, P_k\}$ is a basis for $H^2(X)$ dual to the generators of the semi-group $\Lambda$ of the curve classes in $X$ then in the the equivariant cohomology ring of $X$ the class of the divisor corresponding to the ray $r_j$ is given by $$R_j=\sum_{i=1}^km_{ij}P_i-\lambda_j, \text{   for    } j=1,\dots,N$$  where $\lambda_1,\dots, \lambda_N$ are the equivariant parameters. Note that $n=N-k$, and we can choose the basis $\{P_1,\dots,P_k\}$ so that $(m_{ij})_{j=1,\dots,k}$ is the identity matrix. 
For any $\{i_1,\dots,i_n\} \subset \{1,\dots,N\}$ so that $r_{i_1},\dots,r_{i_n}$ generate a cone in the fan let $$v_{\{i_1,\dots,i_n\}}=R_{i_1}\cdots R_{i_n}$$ be the class of the corresponding fixed point in the equivariant cohomology ring of $X$. Denote by $F$ the set of $\{i_1,\dots,i_n\} \subset \{1,\dots,N\}$ so that $r_{i_1},\dots,r_{i_n}$ generate a cone in the fan. For any $\inn \in F$, let $\nin$ be the equivariant Euler class of the tangent bundle at the corresponding fixed point. This is given by $$\nin=R_{i_1}\cdots R_{i_n}|_{P_1=x_1,\dots,P_k=x_k} $$ where $x_1,\dots,x_k$ uniquely solve the system of equations $$\sum_{i=1}^km_{ij}x_i=\lambda_j \;\;\;\;\;\;\;\ \text{    for    }\;\;\;\;\;\;\;\; j\in\oneton-\inn.$$  Also for the same $\inn \in F$, and any $j\in \inn$ define $${_j}\nin=\frac{R_{i_1}\cdots R_{i_n}}{R_j}|_{P_1=x_1,\dots,P_k=x_k}$$ where  $x_1,\dots,x_k$ are defined as above.

For any $S_1, S_2 \in F$ we have $$v_{S_1}\cdot v_{S_2}=\begin{cases} n_{S_1}v_{S_1} &\text{if} \;\; S_1=S_2 \\ 0 & \text{else}\end{cases}$$ and for any $j\in \oneton$ we have $$R_j=\sum_{\tiny \begin{array}{c} S \in F\\ S \ni j\end{array}}\frac{v_S}{{_j}n_S}$$ in the equivariant cohomology ring. For any $j=1,\dots, N$ define the operator $$D_j=\sum_{i=1}^km_{ij}\frac{\partial}{\partial t_i}-\lambda_j \frac{\partial}{\partial t_0}.$$

The mirror theorem is expressed most simply for the smooth Fano toric variety. In order to make our formulas as explicit as possible for simplicity from now on we assume that $X$ is Fano. For any $\beta \in \Lambda$, let $\beta_i=\int_\beta P_i$, and $R_j(\beta)=\int_\beta R_j$. By the equivariant Mirror theorem \cite{g2} the equivariant small $J$-function of $X$ is given by 
\begin{align*}& J_X(t_0,t_1,\dots,t_k;z)=ze^{(t_0+t_1P_1+\cdots+ t_kP_k)/z}\sum_{\beta\in \Lambda}e^{t_1\beta_1+\cdots+ t_k \beta_k}\prod_{j=1}^N\frac{\prod_{m=-\infty}^{0}(R_j+mz)}{\prod_{m=-\infty}^{R_j(\beta)}(R_j+mz)}.\end{align*} 

We can compute for any nonnegative integer $r$
\begin{align*}& zD_{i_1}\circ\cdots \circ zD_{i_r}J_X(t;z)=ze^{(t_0+t_1P_1+\cdots P_k t_k)/z}\sum_{\beta\in \Lambda}e^{t_1\beta_1+\cdots t_k \beta_k}\times \\&(R_{i_1}+zR_{i_1}(\beta))\cdots(R_{i_r}+zR_{i_r}(\beta))\prod_{j=1}^N\frac{\prod_{m=-\infty}^{0}(R_j+mz)}{\prod_{m=-\infty}^{R_j(\beta)}(R_j+mz)}.\end{align*}

\begin{lem}
If $\dim X \le 3$ then Condition \ref{condition} holds for the operators $D_i$ and the fixed points set basis defined above. 
\end{lem}
\begin{proof}
We prove the case $\dim X =3$, the other cases are similar. It suffices to show that the only positive powers of $z$ appearing in $zD_{i_1}\circ zD_{i_{2}}J_X(t;z)$ is $Az$ for some cohomology class $A$. We claim that the power of $1/z$ in the product $$\prod_{j=1}^N\frac{\prod_{m=-\infty}^{0}(R_j+mz)}{\prod_{m=-\infty}^{R_j(\beta)}(R_j+mz)}$$is at least $2-\delta_{0,R_{i_1}(\beta)}-\delta_{0,R_{i_2}(\beta)}$. For any $1\le j \le N$, the power of $1/z$ in  $$\frac{\prod_{m=-\infty}^{0}(R_j+mz)}{\prod_{m=-\infty}^{R_j(\beta)}(R_j+mz)}$$ is at least $R_j(\beta)$ if $R_j(\beta)$ is nonnegative and it is at least $1+R_j(\beta)$ if $R_j(\beta)$ is negative. If for all $1\le j \le N$, $R_j(\beta)\ge 0$ then clearly the claim holds. If for some $1\le j_0 \le N$ we have$R_{j_0}(\beta)< 0$ then $1+\sum_{j=1}^NR_j(\beta)=1-K_X\cdot \beta \ge 2$ because $X$ is Fano by assumption, so again the claim holds and hence the lemma follows.
\end{proof}

\begin{remark} \label{rem:geometric condition}From the proof of Lemma above one can give the following geometric criterion to ensure that Condition \ref{condition} holds for a general smooth Fano toric variety. Let $$j_X=\min_{C\subset X \text{a rational curve}}(-K_X\cdot \beta+\# \{j|R_j(C)<0\}).$$ Then Condition \ref{condition} holds if $j_X \ge \dim X-1$.
\end{remark}

From now on we assume that $X$ is so that Condition \ref{condition} holds for the operators $D_i$ and the fixed points set basis defined above. Then by the construction $v_{\{i_1,\dots,i_n\}}$ is the coefficient of $z$ in $$zD_{i_1}\circ\cdots \circ zD_{i_n} J(t_0,t_1,\dots,t_k;z)$$  for any $\inn \in F$, and moreover, $$z\nabla_{\vin} J(\tau;z)|_{H^0(X)\oplus H^2(X)}=zD_{i_1}\circ\cdots \circ zD_{i_n} J(t_0,t_1,\dots,t_k;z).$$

For given $S_1, S_2 \in F$, by our computation scheme the $(S_1,S_2)$-entry of $R(t;z_1,z_2)$ is given by 
\begin{align*}
&\la \frac{v_{S_1}}{n_{S_1}},R(t;z_1,z_2)v_{S_2} \ra=-\delta^{S_1}_{S_2}+\frac{1}{z_1^{-1}+z_2^{-1}}\times \\&\sum_{\inn \in F} \frac{n_{S_2}}{n_{\inn}}[zD_{i_1}\circ\cdots \circ zD_{i_n} J(t;z)]_{v_{S_1}}[zD_{i_1}\circ\cdots \circ zD_{i_n} J(t;z)]_{v_{S_2}}\end{align*} where $[-]_{v_S}$ is the coordinate along the basis element $v_S$.

For any $S \in F$ introduce the variables $X_S$, $Y_S$ with the relations $$X_{S_1}X_{S_2}=\begin{cases} n_{S_1} X_{S_1} &\text{if} \;\; S_1=S_2 \\ 0 & \text{else}\end{cases}\quad \text{and} \quad Y_{S_1} Y_{S_2}=\begin{cases} n_{S_1} Y_{S_1} &\text{if} \;\; S_1=S_2 \\ 0 & \text{else.}\end{cases}$$ Moreover, for any $j\in \oneton$ define the new variables $U_j$ and $V_j$ by $$U_j=\sum_{\tiny \begin{array}{c} S \in F\\ S \ni j\end{array}}\frac{X_{S}}{{_j}n_S}\quad \text{and}\quad 
V_j=\sum_{\tiny \begin{array}{c} S \in F\\ S \ni j\end{array}}\frac{Y_{S}}{{_j}n_S}.$$

Then for $\beta \in \Lambda- \{0\}$ we get 
\begin{align*}
&\sum_{S_1,S_2\in F}\la \frac{v_{S_1}}{n_{S_1} (z_1-\psi_1)},\frac{v_{S_2}}{n_{S_2} (z_2-\psi_2)} \ra_{0,\beta} X_{S_1}Y_{S_2}=
\\&\frac{1}{z_1+z_2}\sum_{\inn \in F}\frac{1}{n_{\inn}}\sum_{\beta_1+\beta_2=\beta}\prod_{j=1}^n(U_{i_j}+zR_{i_j}(\beta_1))(V_{i_j}+zR_{i_j}(\beta_2))\times\\&\prod_{r=1}^N\frac{\prod_{m=-\infty}^{0}(U_r+mz_1)(V_r+mz_2)}{\prod_{m=-\infty}^{R_r(\beta_1)}(U_r+mz_1)\prod_{m=-\infty}^{R_r(\beta_2)}(V_r+mz_2)}.
\end{align*}

\subsection*{Example: $X=\P^n$}

In this case $N=n+1$, and $H^2(X)$ is generated by the hyperplane class denoted by $P$. It can be easily seen that $X$ satisfies the condition in Remark \ref{rem:geometric condition}. According \cite[Theorem 9.5]{g1} the equivariant $J$-function of $X$ is $$J(t_0,t_1;z)=ze^{(t_0+Pt_1)/z}\sum_{d=0}^{\infty}e^{dt_1}\frac{1}{\prod_{m=1}^d(R_1+mz)\cdots(R_{n+1}+mz)}$$ where $R_j=P-\lambda_j$. In this case $D_j=\frac{\partial}{\partial t_1}-\lambda_j \frac{\partial}{\partial t_0}$, and one can compute

$$zD_{i_1}\circ\cdots \circ zD_{i_n} J(t_0,t_1,;z)=ze^{(t_0+Pt_1)/z}\sum_{d=0}^{\infty}e^{dt_1}\frac{(R_{i_1}-dz)\cdots (R_{i_n}-dz)}{\prod_{m=1}^d(R_1+mz)\cdots(R_{n+1}+mz)}.$$
In this case $F$ is the set of the subets of $\{1,\dots,n+1\}$ with $n$ elements. For any $S \in F$ let $s \in \{1,\dots,n+1\}-S$, then $n_S=\prod_{i \in S}(\lambda_{s}-\lambda_{i})$. %Now for any $S_1, S_2 \in F$ from the we know \begin{align*} \la \frac{v_{S_2}}{n_{S_2}},S(t_1P;z)^*v_{S_1})\ra &=z^{-1}[D_{i_1}\circ\cdots \circ D_{i_n} J(t_0,t_1;z)]_{\vjn}\\&=\frac{1}{\njn}\la \vin,S(t_1P;z)\vjn)\ra \end{align*} where $[-]_{\vin}$ is the coefficient along $\vin$. So $$\la \frac{\vin}{\nin},S(t_1P;z)\vjn)\ra=\frac{\njn}{\nin}z^{-1}[D_{i_1}\circ\cdots \circ D_{i_n} J(t_0,t_1,\dots,t_k;z)]_{\vjn}.$$
%So the $(\sin,\sjn)$-entry  of  $R(t;z_1,z_2)$ is given by 
%\begin{align*} \label{equ:R-entry}
%&\la \frac{\vin}{\nin},R(t,z_1,z_2)\vjn \ra=-\delta^{\sjn}_{\sin}+\frac{1}{z_1^{-1}+z_2^{-1}}\times \\&\sum_{\tiny \begin{array}{l}\kn \subset\\ \nplusone\end{array}} \frac{\njn}{n_{\kn}}[D_{k_1}\circ\cdots \circ D_{k_n} J(t;z)]_{\vin}[D_{k_1}\circ\cdots \circ D_{k_n} J(t;z)]_{\vjn}.\end{align*}
For $d>0$ we get 
\begin{align*}
&\sum_{S_1,S_2 \in F} \la \frac{v_{S_1}}{n_{S_1}( z_1-\psi_1)},\frac{v_{S_2}}{n_{S_2} (z_2-\psi_2)} \ra_{0,d} X_{S_1}Y_{S_2}=\\&
\frac{1}{z_1+z_2}\sum_{\inn \in F}\frac{1}{n_{\inn}}\times \\&\sum_{d_1+d_2=d}\frac{(U_{i_1}-d_1z_1)\cdots (U_{i_n}-d_1z_1)(V_{i_1}-d_2z_2)\cdots (V_{i_n}-d_2z_2)}{\prod_{m=1}^{d_1}(U_1+mz_1)\cdots(U_{n+1}+mz_1)\prod_{m=1}^{d_2}(V_1+mz_2)\cdots(V_{n+1}+mz_2)}.
\end{align*}

%We know that for any $l=0,\dots, n$, $\displaystyle P^l=\sum_{S \in F}\frac{\lambda_{s}^l}{n_S}v_S,$ so 

%\begin{align*}
%&\la \frac{P^{l_1}}{z_1-\psi_1},\frac{P^{l_2}}{z_2-\psi_2} \ra_{0,d} =\lim_{\lambda_0,\dots ,\lambda_n \to 0}
%\\ &  \sum_{\tiny\begin{array}{c}\inn\\ \jn \end{array} \subset \nplusone} \lambda_{\sin}^{l_1} \lambda_{\sjn}^{l_2}\la \frac{\vin}{\nin( z_1-\psi_1)},\frac{\vjn}{\njn (z_2-\psi_2)} \ra_{0,d}.
%\end{align*}
%The limit in the right hand side exists as the left hand side has a nonequivariant limit.
\begin{remark} 
We know that for any $l=0,\dots, n$  $$P^l=\sum_{S \in F}\frac{\lambda_{s}^l}{n_S}v_S,$$ so one can get the ordinary two-point invariants $$\la \frac{P^{l_1}}{z_1-\psi_1},\frac{P^{l_2}}{z_2-\psi_2} \ra_{0,d}$$ from the equivariant two-point invariants above by taking the non-equivariant limits. 
\end{remark}
\begin{remark}
This example can be easily generalized to the case $X=\P^{n_1}\times\cdots \times \P^{n_k}$ for $n_1,\dots, n_k\in \mathbb{Z}_{>0}$. For this first note that $X$ staisfies the condition in Remark \ref{rem:geometric condition}. Now if $P_1,\dots P_k \in H^2(X)$ are the pullbacks of of the hypeplane class from each factor then for any $1\le r\le k$ one can take $$R_{j_r}=P_r-\lambda_{j_r} \;\;\; \text{for} \;\;\;  1\le j_r\le n_r+1$$ and proceed as above to recover the formula in \cite[Theorem 1.1]{p}. 
\end{remark}
\subsection*{Examples of semi-Fano toric manifolds}
We first consider the toric manifold $X_1=\P(\O_{\P^1}\oplus\O_{\P^1}(1)\oplus\O_{\P^1}(1))$. To see if our method works here 
we check the Condition \ref{condition} directly.  Let $P_1$ and $P_2$ be respectively the class of fiber and the unversal divisor on $X$. Then the class of the equivairant divisors are $R_1=P_1-\lambda_1$, $R_2=P_1-\lambda_2$, $R_3=P_2-\lambda_3$, $R_4=P_2-P_1-\lambda_4$, and $R_5=P_2-P_1-\lambda_5$. 
By \cite{g2} the equivariant $I$-function is 
\begin{align*}&ze^{t_0+t_1P_1+t_2P_2/z}\sum_{d_1,d_2=0}^\infty e^{t_1d_1+t_2d_2}\times \\&\frac{\prod_{m=-\infty}^0(R_4+mz)(R_5+mz)}{\prod_{m=1}^{d_1}(R_1+mz)(R_2+mz)\prod_{m=1}^{d_2}(R_3+mz)\prod_{m=-\infty}^{d_2-d_1}(R_4+mz)(R_5+mz)}\end{align*} and coincides with the equivariant small $J$-function.

We have $$D_1=\frac{\pr}{\pr t_1}-\lambda_1\frac{\pr}{\pr t_0}, D_3=\frac{\pr}{\pr t_2}-\lambda_3\frac{\pr}{\pr t_0}, D_4=(\frac{\pr}{\pr t_2}-\frac{\pr}{\pr t_1})-\lambda_4\frac{\pr}{\pr t_0}\;\; \text{etc.}$$ We denote the fraction in the sum above by $I(d_1,d_2)$. If $d_2 \ge d_1$ then up to a constant factor $I(d_1,d_2)$ is $1/z^{3d_2}(1+o(1/z))$, and if $d_2<d_1$ then up to a constant factor $I(d_1,d_2)$ is $1/z^{3d_2+2}(1+o(1/z))$.

By symmetry we only need to check Condition \ref{condition} for  $zD_3\circ zD_1 I$ and $zD_4 \circ zD_1 I$. We can then conclude that $$z\nabla_{R_4R_3R_1}J=zD_4\circ zD_3\circ zD_1 I\quad \text{and}\quad z\nabla_{R_5R_3R_1}J=zD_5\circ zD_3\circ zD_1 I \quad \text{etc.}$$ and so we can follow the rest of the computations in Section \ref{subsec:fixed_pt} for $X_1$ without change.
  We have $$zD_3\circ zD_1 I=ze^{t_0+p_1t_1+p_2t_2/z}\sum_{d_1,d_2=0}^\infty e^{t_1d_1+t_2d_2}(R_1+d_1z)(R_3+d_2z)I(d_1,d_2).$$ By comparing the power of $z$ in $(R_1+d_1z)(R_3+d_2z)$ to the power of $1/z$ in $I(d_1,d_2)$ we see that the condition \ref{condition} holds. A similar analysis shows that the same holds for $zD_4\circ  zD_1 I$.

The second example we consider is $X_2=\P(\O_{\P^1} \oplus\O_{\P^1} \oplus \O_{\P^1}(2))$. As in the previous example we demonstrate how we can check Condition \ref{condition}. Again let $P_1$ and $P_2$ be respectively the class of fiber and the unversal divisor on $X$. Then the class of the equivairant divisors are $R_1=P_1-\lambda_1$, $R_2=P_1-\lambda_2$, $R_3=P_2-\lambda_3$, $R_4=P_2-\lambda_4$, and $R_5=P_2-2P_1-\lambda_5$. By \cite{g2} the equivariant $I$-function is 
\begin{align*}&ze^{t_0+t_1P_1+t_2P_2/z}\sum_{d_1,d_2=0}^\infty e^{t_1d_1+t_2d_2}\times \\&\frac{\prod_{m=\infty}^0(R_5+mz)}{\prod_{m=1}^{d_1}(R_1+mz)(R_2+mz)\prod_{m=1}^{d_2}(R_3+mz)(R_4+mz)\prod_{m=-\infty}^{d_2-2d_1}(R_5+mz)}.\end{align*} This time the miorror transformation involves a nontrivial change of variables. Following the notation in \cite{g2} let $$f(Q)=\sum_{d=1}^\infty\frac{(2d-1)!}{(d!^2)}Q^d.$$ Then we modify our derivation operators in Section \ref{subsec:fixed_pt} according to the mirror transformation: 
\begin{align*}
D_1&=\frac{\pr}{\pr T_1}-\lambda_1\frac{\pr}{\pr t_0}, D_2=\frac{\pr}{\pr T_1}-\lambda_2\frac{\pr}{\pr t_0}, D_3=\frac{\pr}{\pr T_2}-\lambda_3\frac{\pr}{\pr t_0}, D_4=\frac{\pr}{\pr T_2}-\lambda_4\frac{\pr}{\pr t_0}, \\D_5&=\frac{\pr}{\pr T_2}-2\frac{\pr}{\pr T_1}-\lambda_5\frac{\pr}{\pr t_0},\end{align*} where $$\frac{\pr}{\pr T_1}=\frac{1}{1+2e^{t_1}f'(e^{t_1})}\frac{\pr}{\pr t_1}\;\; \text{and}\;\; \frac{\pr}{\pr T_2}=\frac{1}{1-e^{t_2}f'(e^{t_1})}\frac{\pr}{\pr t_2}.$$  

We denote the fraction in the sum above by $I(d_1,d_2)$. If $d_2 \ge 2d_1$ then up to a constant factor $I(d_1,d_2)$ is $1/z^{3d_2}(1+o(1/z))$, and if $d_2<2d_1$ then up to a constant factor $I(d_1,d_2)$ is $1/z^{3d_2+1}(1+o(1/z))$. By symmetry we only check Condition \ref{condition} for $zD_3\circ zD_1 I$ and $zD_4 \circ zD_1 I$. After this we can conclude that $$z\nabla_{R_4R_3R_1}J=zD_4\circ zD_3\circ zD_1 I\quad \text{and} \quad z\nabla_{R_5R_3R_1}J=zD_5\circ zD_3\circ zD_1 I\quad \text{etc.}$$ after the change of variables (see \cite{g2} for the details). So we can follow the rest of computations of Section \ref{subsec:fixed_pt} for $X_2$ as well. We have \begin{align*}&zD_3\circ zD_1 I=\frac{ze^{t_0+P_1t_1+P_2t_2/z}}{(1+2e^{t_1}f'(e^{t_1}))(1-e^{t_2}f'(e^{t_1}))}\sum_{d_1,d_2=0}^\infty e^{t_1d_1+t_2d_2}\times \\&(P_1-(1+2e^{t_1}f'(e^{t_1}))\lambda_1+d_1z)(P_2-(1-e^{t_2}f'(e^{t_1}))\lambda_3+d_2z)I(d_1,d_2).\end{align*} By comparing the power of $1/z$ in $I(d_1,d_2)$ to the power of $z$ in the factor of $I(d_1,d_2) $we can again verify Condition \ref{condition}. Similar analysis shows that the same is true for $zD_4 \circ zD_1 I$.

\begin{remark}
Note that Condition \ref{condition} does not hold for $zD_5\circ zD_1 I$. In fact  \begin{align*} &zD_5\circ zD_1 I=\frac{ze^{t_0+P_1t_1+P_2t_2/z}}{1+2e^{t_1}f'(e^{t_1})}\sum_{d_1,d_2=0}^\infty e^{t_1d_1+t_2d_2}\left(P_1-(1+2e^{t_1}f'(e^{t_1}))\lambda_1+d_1z\right)\times\\&\left(\frac{P_2}{1-e^{t_2}f'(e^{t_1})}-\frac{2P_1}{1+2e^{t_1}f'(e^{t_1})}-\lambda_5+(\frac{d_2}{1-e^{t_2}f'(e^{t_1})}-\frac{2d_1}{1+2e^{t_1}f'(e^{t_1})})z\right)I(d_1,d_2)\\&+\text{other terms}.\end{align*} and one can see that in $I(d_1,0)$ for $d_1>0$ there are terms of $z$-degree equal to -1, and the factor of $I(d_1,0)$ has terms of $z$-degrees 2. This means that $z\nabla_{R_4R_5R_1}J \neq zD_4\circ zD_5\circ zD_1 I$ whereas by the last paragraph $z\nabla_{R_5R_4R_1}J=zD_5\circ zD_4\circ zD_1 I$.
\end{remark}

%%%%%%%%%%%%%%%%%%%%%%%%%%%%%%%%%%%

\end{document}